\documentclass[12pt]{amsart}

\usepackage{amsmath, srcltx}
\usepackage{amssymb}
\usepackage{amsfonts,cmtiup,comment}

\usepackage{mathrsfs,cmtiup}
\usepackage{amsthm, eucal, eufrak}

\usepackage{srcltx}
\usepackage{amsfonts,amssymb,mathrsfs,amsmath,cmtiup}
\usepackage{amsthm, eucal, eufrak}

\makeatletter
\def\blfootnote{\xdef\@thefnmark{}\@footnotetext}
\makeatother

\usepackage{color}

\newtheorem{theorem}{Theorem}[section]
\newtheorem{lemma}[theorem]{Lemma}
\newtheorem{proposition}[theorem]{Proposition}
\newtheorem{corollary}[theorem]{Corollary}

\newtheorem{problem}[theorem]{Problem}

\theoremstyle{definition}

\let\leq=\leqslant
\let\geq=\geqslant

\begin{document}

\title{\boldmath Nonsoluble and non-$p$-soluble length\\ of finite groups}
%\thanks{....}

\author{E. I. Khukhro}

\address{Sobolev Institute of Mathematics, Novosibirsk, 630\,090,
Russia}
 \email{khukhro@yahoo.co.uk}

\author{P. Shumyatsky}

\address{Department of Mathematics, University of Brasilia, DF~70910-900, Brazil}

\email{pavel@unb.br}


%\maketitle

\begin{abstract}
Every finite group $G$ has a normal series each of whose factors either is soluble or is a direct product of nonabelian simple groups. We define the nonsoluble length $\lambda (G)$  as the minimum number of nonsoluble factors in a series of this kind. Upper bounds for $\lambda (G)$  appear in the study of various problems on finite, residually finite, and profinite groups.
We prove that $\lambda (G)$ is bounded in terms of the maximum $2$-length of soluble subgroups of $G$, and that $\lambda (G)$
is bounded by the maximum Fitting height of soluble subgroups. For an odd prime $p$,  the  non-$p$-soluble length $\lambda _p(G)$ is introduced, and it is proved that $\lambda _p(G)$ does not exceed
the maximum $p$-length of $p$-soluble subgroups. We conjecture that for a given prime $p$ and a given proper group variety ${\frak V}$ the  non-$p$-soluble length $\lambda _p(G)$  of finite groups $G$ whose Sylow $p$-subgroups belong to ${\frak V}$ is bounded.
   In this paper we prove this conjecture for any variety
   that is a product of several soluble varieties and varieties of finite exponent.
  As an application of the results obtained, an error is corrected  in the proof of the main result of the second author's paper  ``Multilinear commutators in residually finite groups'', \emph{Israel J. Math.} \textbf{189} (2012), 207--224.
\end{abstract}

\maketitle

\blfootnote{This work was supported by CNPq-Brazil. The first author thanks  CNPq-Brazil and the University of Brasilia for support and hospitality that he enjoyed during his visits to Brasilia.}

\section{Introduction}

Every finite group $G$ has a normal series each of whose factors either is soluble or is a direct product of nonabelian simple groups. We define the \emph{nonsoluble length} of $G$, denoted by $\lambda (G)$,  as the minimum number of nonsoluble factors in a series of this kind: if
$$
1= G_0\leq G_1\leq \dots \leq G_{2h+1}=G
$$
is a shortest normal series in which  for $i$  odd  the factor $G_{i+1}/G_{i}$ is soluble (possibly trivial), and for $i$ even   the factor $G_{i+1}/G_{i}$   is a (non-empty) direct product of nonabelian simple groups, then the nonsoluble length $\lambda (G)$ is equal to $h$.  For any prime $p$, we introduce a similar notion of non-$p$-soluble length $\lambda _p (G)$  by replacing ``soluble'' by ``$p$-soluble''.  Recall that a finite group is said to be \emph{$p$-soluble} if it has a normal series each of whose factors is either a $p$-group or a $p'$-group; the least number of $p$-factors in such a series is called the \emph{$p$-length} of the group.
 Of course, $\lambda (G)= \lambda_2 (G)$, since groups of odd order are soluble by the Feit--Thompson theorem \cite{fei-tho}.

 Upper bounds for the nonsoluble and non-$p$-soluble length appear in the study of various problems on finite, residually finite, and profinite groups. For example, such  bounds were implicitly obtained in the Hall--Higman paper \cite{ha-hi} as part of their reduction of the Restricted Burnside Problem to $p$-groups. Such bound were also a part of Wilson's  theorem \cite{wil83} reducing the problem of local finiteness of periodic profinite groups to pro-$p$-groups. (Both problems were solved by Zelmanov \cite{zel89, zel90, zel91, zel92}). More recently, bounds for the nonsoluble length were needed in the study of verbal subgroups in finite groups \cite{dms1, 68, austral}.

 In the present paper we show that information on soluble ($p$-soluble) subgroups of a finite group $G$ can be used for obtaining upper bounds for the nonsoluble (non-$p$-soluble) length of $G$.  We prove that the nonsoluble length $\lambda (G)$ is bounded in terms of the maximum $2$-length of soluble subgroups of $G$, and for $p\ne 2$ the non-$p$-soluble length $\lambda _p(G)$ is bounded by the maximum $p$-length of $p$-soluble subgroups.

 \begin{theorem}\label{t-2l}
\emph{(a)} The nonsoluble length $\lambda (G)$ of a finite group $G$  does not exceed $2L_2+1$, where $L_2$ is the maximum $2$-length of soluble subgroups of $G$.

\emph{(b)} For $p\ne 2$, the non-$p$-soluble length $\lambda_p (G)$ of a finite group $G$  does not exceed the maximum $p$-length of $p$-soluble subgroups of $G$.
 \end{theorem}

     A bound for the non-$p$-soluble length $\lambda_p (G)$ of a finite group in terms of the maximum $p$-length of its $p$-soluble subgroups can be achieved by similar arguments as in the proof of  part (a) of Theorem~\ref{t-2l}. We state part (b) separately, with a
different proof, in order to achieve a better bound.

As a consequence of the proof  of Theorem~\ref{t-2l}(a) we obtain the following bound for the nonsoluble length in terms of the maximum Fitting height of soluble subgroups.

 \begin{corollary}
 \label{c-fh}
The nonsoluble length $\lambda (G)$ of a finite group $G$  does not exceed the maximum Fitting height of soluble subgroups of $G$.
 \end{corollary}

Obviously,  Theorem~\ref{t-2l} can be applied in all situations where the $p$-length of $p$-soluble subgroups is known to be bounded. For example, such bounds are given by the well-known theorems
when either the exponent, or the derived length of a Sylow $p$-subgroup is bounded. For $p\ne 2$, these are the original Hall--Higman
theorems \cite{ha-hi}, and for $p=2$ the results by Hoare \cite{hoa}, Gross \cite{gro}, Berger and Gross \cite{ber-gro}, and, with best possible bounds, by Bryukhanova \cite{bry79, bry81}.

There is a long-standing problem on $p$-length due to Wilson (problem 9.68 in Kourovka Notebook \cite{kour}): \emph{for a given prime $p$ and a given proper group variety ${\frak V}$, is there a bound for the $p$-length of finite $p$-soluble groups whose Sylow $p$-subgroups belong to ${\frak V}$?}

We state here a problem on non-$p$-soluble length by analogy with Wilson's problem.

\begin{problem}\label{prob}
For a given prime $p$ and a given proper group variety ${\frak V}$, is there a bound for the non-$p$-soluble length $\lambda _p$ of finite groups whose Sylow $p$-subgroups belong to ${\frak V}$?
\end{problem}

By Theorem~\ref{t-2l} an affirmative answer to Problem~\ref{prob} would follow from an affirmative answer to the aforementioned Wilson's problem. Wilson's problem  so far has seen little progress beyond the aforementioned affirmative answers for soluble varieties and varieties of bounded exponent (and, implicit in the Hall--Higman theorems \cite{ha-hi}, for ($n$-Engel)-by-(finite exponent) varieties). The next step would be some combination of solubility and exponent, but, for example, Wilson's problem remains open for (finite exponent)-by-soluble  varieties (but known for  soluble-by-(finite exponent)). Our Problem~\ref{prob} may be more tractable. In the present paper we obtain an affirmative answer in the case of any variety that is a product of several soluble varieties and varieties of finite exponent.
Recall that the product ${\frak V}{\frak W}$ of two varieties is the variety of groups having a normal subgroup in  ${\frak V}$ with quotient in ${\frak W}$; this is  generalized to more than two factors in obvious fashion. The variety of groups of exponent $n$ is denoted by ${\frak B}_{n}$,
and the variety of soluble groups of derived length $d$ by ${\frak A}^{d}$.

 \begin{theorem}\label{t-exp-sol} Let $p$ be a prime and let $\frak V$ be a variety that is a product of several soluble varieties and varieties of finite exponent. Then the non-$p$-soluble length $\lambda _p(G)$ of finite groups whose Sylow $p$-subgroups belong to ${\frak V}$ is bounded. More precisely, if a Sylow $p$-subgroup of a finite group $G$ belongs to the variety
${\frak B}_{p^{a_1}}{\frak A}^{d_1}\cdots {\frak B}_{p^{a_n}}{\frak A}^{d_n}$ for some integers $a_i,d_i\geq 0$, then the non-$p$-soluble length $\lambda _p(G)$ is bounded in terms of $\sum a_i+\sum d_i$.
 \end{theorem}

We did not exclude trivial factors for greater generality of the statement; so in fact the first (nontrivial) factor may be either a variety of finite exponent, or a soluble variety.

The proofs in this paper use the classification of finite simple groups in its application to Schreier's Conjecture, that the outer automorphism groups of finite simple groups are soluble.

At the end of the paper, we apply Corollary~\ref{c-fh} for correcting an error  in the proof of one of the lemmas in the second author's paper  \cite{896}, which also has a bearing on the papers \cite{dms1,austral}.  It was proved in \cite{896}  that under the hypotheses of the main theorem all soluble subgroups of a finite group have bounded Fitting height. This is why Corollary~\ref{c-fh} can be applied to provide an alternative  proof of the result in \cite{896}.

\section{Preliminaries}

 Let $p$ be a prime, and $G$ a finite group. Let
$$
1= G_0\leq G_1\leq \dots \leq G_{2h+1}=G
$$
be a shortest series such that $G_{i+1}/G_{i}$ is $p$-soluble (possibly trivial) for $i$ odd, and is a (non-empty) direct product of nonabelian simple groups of order divisible by $p$ for $i$ even. Then $h$ is defined to be the \emph{non-$p$-soluble length} of $G$,  denoted by $\lambda_p (G)$. For $p=2$ we speak about the \emph{nonsoluble length} denoted by  $\lambda (G)= \lambda_2 (G)$. It is easy to see that the  non-$p$-soluble length behaves well under taking normal subgroups, homomorphic images, and (sub)direct products. It is also clear that an extension of a normal subgroup of non-$p$-soluble length $k$ by a group of non-$p$-soluble length $l$ has non-$p$-soluble length at most $k+l$.

The soluble radical of a group $G$, the largest normal soluble subgroup, is denoted by $R(G)$.  The largest normal $p$-soluble subgroup is called the  \emph{$p$-soluble radical} denoted by  $R_p(G)$.

Consider the quotient $\bar G=G/R_p(G)$ of a finite group $G$ by it $p$-soluble radical. The socle $Soc(\bar G)$, that is, the product of all minimal normal subgroups of $\bar G$, is a direct product  $Soc(\bar G)=S_1\times\dots\times S_m$ of nonabelian simple groups $S_i$ of order divisible by $p$.
The group $G$ induces by conjugation a permutational action on the set $\{S_1, \dots , S_m\}$. Let $K_p(G)$ denote the kernel of this action, which we give a technical name of the \emph{$p$-kernel subgroup} of $G$; of course, $K_p(G)$ is the full inverse image in $G$ of $\bigcap N_{\bar G}(S_i)$. For $p=2$ we speak of the \emph{kernel subgroup} of $G$ denoted by $K(G)$.

\begin{lemma}\label{l-kern}
The $p$-kernel subgroup $K_p(G)$ has non-$p$-soluble length at most 1.
\end{lemma}

\begin{proof}
We can assume that $R_p(G)=1$. Let $K=K_p(G)$ for short. For every $i$, clearly, $K/C_K(S_i)$ embeds in the automorphism group of the nonabelian simple group $S_i$. Hence $K/C_K(S_i)S_i$ is soluble by Schreier's Conjecture as a consequence of the classification.  The result follows, since $K/\bigcap C_K(S_i)$ embeds in the direct product of the $K/C_K(S_i)$ and $\bigcap C_K(S_i)=1$ because $R_p(G)=\nobreak 1$.
\end{proof}

Given a prime $p$ and a finite group $G$, we define higher $p$-kernel subgroups by induction: $K_{p,1}(G)= K_{p}(G)$ and $K_{p,i+1}(G)$ is the full inverse image of $K_{p}(G/K_{p,i}(G))$. For $p=2$ we speak of higher kernel subgroups denoted by $K_{i}(G)=K_{2,i}(G)$. We may drop the indication of the group when it causes no confusion writing simply $K_{p,i}$ and $K_{i}$. Obviously, the non-$p$-soluble length of $K_{p,i}$ is $i$.

\section{Non-$p$-soluble length of a group and $p$-length of soluble subgroups}

First we prove Theorem~\ref{t-2l}(a) concerning nonsoluble length and $2$-length.
Rather than working with $2$-length of soluble subgroups, we found it expedient to introduce a slightly different technical parameter.
Every soluble finite group $G$ has a normal series each of whose factors is a direct product of a $2$-group and a $2'$-group, of which either may be trivial, but not both; we call such a series a \emph{$(2\times 2')$-series}.  We define the \emph{$(2\times 2')$-length} of $G$ to be the length of a shortest series of this kind and denote it by $l_{2\times 2'}$.
Obviously, the $(2\times 2')$-length of a subgroup or a quotient does not exceed the $(2\times 2')$-length of a soluble group; this parameter also behaves well under passing to (sub)direct products.
If $l_2$ is the ordinary $2$-length of $G$, then obviously $l_{2\times 2'}\leq 2l_2+1$, since, say, an upper $2$-series is also a $(2\times 2')$-series.

Now let $G$ be an arbitrary (not necessarily soluble) finite group. Let $L_{2\times 2'}(G)$ denote the maximum $({2\times 2'})$-length of soluble subgroups of $G$. This parameter also behaves well under passing to subgroups, homomorphic images, and (sub)direct products.

\begin{lemma}\label{facto}
If $N$ is a normal subgroup of a finite group $G$, then $L_{2\times 2'} (G/N)\leq L_{2\times 2'} (G)$.
\end{lemma}

\begin{proof}  Let $K/N$ be a soluble subgroup of $G/N$.
We wish to show that $l_{2\times 2'}(K/N)\leq L_{2\times 2'}(G)$.
Let $T$ be a Sylow 2-subgroup of $N$ and let $H=N_K(T)$. Then $K=NH$ by the Frattini argument. Using the Feit--Thompson theorem \cite{fei-tho} it is easy to see that $H$ is a soluble subgroup of $G$.
Since $K/N$ is a homomorphic image of $H$, we have $l_{2\times 2'}(K/N)\leq l_{2\times 2'}(H)\leq L_{2\times 2'}(G)$.
\end{proof}

Given a finite soluble group $G$, we define the \emph{lower $(2\times 2')$-series} by induction: let $D_1(G)=G$ and let $D_{i+1}(G)$ be the smallest normal subgroup of $D_{i}(G)$ such that $D_i/D_{i+1}$ is a direct product of a $2$-group and a $2'$-group. Clearly, the $D_i(G)$ are characteristic subgroups of $G$. If $h=L_{2\times 2'}(G)$, then $D_{h+1}=1$ and $D_h$ is a direct product of a $2$-group and a $2'$-group.
It is easy to see that if $H$ is a subgroup of $G$, then $D_i(H)\leq D_i(G)$, and if $H$ is a normal subgroup of $G$, then $D_i(G/H)\leq D_i(G)H/H$.

\begin{lemma}\label{sisis}
Let $G$ be a finite group
and let $V$ be a non-empty  product of nonabelian minimal normal subgroups of $G$. Let $V=S_1\times S_2\times\cdots\times S_m$, where the $S_i$ are simple factors of $V$. Let $K=\bigcap_i N_G(S_i)$. Then $L_{2\times 2'}(G/K)<L_{2\times 2'}(G)$.
\end{lemma}

\begin{proof} Let $G$ be a counter-example of minimal possible order. Suppose that  $V=V_1\times V_2$, where $V_1=S_1\times \cdots\times S_k$ and $V_2=S_{k+1}\times \cdots\times S_m$ are normal subgroups of $G$. Let $K_1=\bigcap_{i=1}^k N_G(S_i)$ and $K_2=\bigcap_{i=k+1}^m N_G(S_i)$. It is easy to see that the subgroup $K_j/V_i$ plays the same role for $G/V_i$ and $V_jV_i/V_i$ for $\{i,j\}=\{1,2\}$ as $K$ for $G$ and $V$. By induction, $L_{2\times 2'}(G/K_j)<L_{2\times 2'}(G/V_i)\leq L_{2\times 2'}(G)$ by Lemma \ref{facto}. Since $G/K$ embeds in $G/K_1\times G/K_2$, this contradicts the assumption that $G$ is a counter-example. Hence $V$ is a minimal normal subgroup in $G$, and acting by conjugation $G$ transitively permutes  $S_1,  S_2, \dots , S_m$. By a similar argument, the minimality of $G$ also implies that $C_G(V)=1$. %sic! nado dlia K/V sol.
Hence $K/V$ is soluble by the Schreier Conjecture.

Let $L_{2\times 2'} (G)=h$. If $H/K$ is a soluble subgroup  of $G/K$,  Lemma \ref{facto} shows that $L_{2\times 2'} (H/K)\leq h$. Since $G$ is a counterexample, there is such a subgroup with $L_{2\times 2'} (H/K)=h$. By minimality of $|G|$, we have $G=H$, so that $G/K$ is soluble. Since $K/V$ is soluble, we conclude that $G/V$ is soluble.

In order to avoid repeating the same arguments, we now choose in a special way a prime $q$ dividing  $|V|$. Namely,  if $D_h(G/K)$ is not  $2$-group, then  we put $q=2$ using the fact that $V$ is nonsoluble. If, however,  $D_h(G/K)$ is a $2$-group, then we choose an odd divisor $q$ of $|V|$. Choose a Sylow $q$-subgroup $Q$ in $V$ and let $T=N_G(Q)$ and $N=N_V(Q)$. By the Frattini argument $G=VT$. Write $Q_i=Q\cap S_i$ and $N_i=N_{S_i}(Q_i)$ for $i=1,\dots,m$. Then  $N=N_1\cdots N_m$.  Suppose that $T$ is nonsoluble. Since  $G/V$ is soluble  by the above, then $N$ is nonsoluble.

Let $R_i=R(N_i)$ for $i=1,\dots,m$ and $R=R(N)=\prod_i R_i$ be the corresponding soluble radicals. Choose a minimal normal subgroup $B/R$ of $T/R$ contained in $N/R$. Since $R(N/R)=1$, it follows that $B/R$ is a direct product of nonabelian simple groups. Since $B/R\leq N/R\cong N_1/R_1\times\dots\times N_m/R_m$, it follows that there are subgroups  $B_i\leq N_i$ such that $B=B_1\times\dots \times B_m$ and  $B_iR/R$ is a direct product of nonabelian simple groups.  Note that $T$ transitively permutes the factors $B_iR/R$, since $T$ acts transitively on the set of $S_i$. Let now $x\in T$ normalize $B_iR$. Since $R$ is soluble, it follows that $B_i^x\leq N_i$ and therefore $x$ normalizes $S_i$.

 We now apply the induction argument to $T/R$  and its minimal normal subgroup $B/R$, which we represent as a direct product $B/R=U_1\times\dots\times U_n$, where the $U_i$ are nonabelian simple groups, each of which must be contained in one of the $B_iR/R$. It follows that  the subgroup  $C=\bigcap N_T(U_i)$ is contained in $\bigcap N_T(B_iR)$, which is in turn contained in $K\cap T$ by the above argument.
Since $T$ is a proper subgroup of $G$,  we deduce that $L_{2\times 2'}(T/C)\leq h-1$. Then also $L_{2\times 2'}(T/(K\cap T)\leq h-1$ and $L_{2\times 2'}(G/K)\leq h-1$, a contradiction with the assumption.

Thus, we can assume that $T$ is soluble. Since $G/K$ is a homomorphic image of $T$, it follows that $L_{2\times 2'}(T)=h$,  so that $D_h(T)$ is a direct product of a $2$-group by a $2'$-group. If $q=2$, then  let $P$ be the Hall $2'$-subgroup of  $D_h(T)$, and if $q\ne 2$, then let $P$ be the Sylow $2$-subgroup of $D_h(T)$. In either case, $P$ is normal in $T=N_G(Q)$ and therefore $P$ and $Q$ commute. In particular, $P$ centralizes $Q_i$ for every $i=1,\dots,m$. But then $P$ must normalize all the factors $S_1,S_2,\ldots,S_m$ and so $P\leq K$. If $q=2$, then this contradicts the assumption that  $D_h(G/K)\leq  D_h(T)K/K$  is not a  $2$-group. If $q\ne 2$, then this contradicts the assumption that  $D_h(G/K)\leq D_h(T)K/K$  is a  $2$-group.
\end{proof}

\begin{proof}[Proof of Theorem~\ref{t-2l}]  Recall that  $L_2$ is the maximum $2$-length of soluble subgroups of a finite group $G$. We need to prove that the nonsoluble length  $\lambda (G)$  does not exceed $2L_2+1$. Since $L_{2\times 2'}(G)\leq 2L_2+1$, it is sufficient to prove he following.

 \begin{theorem}\label{t-2ll}
 The nonsoluble length $\lambda (G)$   does not exceed $L_{2\times 2'}(G)$.
 \end{theorem}

 \begin{proof}
 We proceed by induction on $L_{2\times 2'}(G)$. Using Lemma~\ref{facto} we can clearly assume that $G$ has trivial soluble radical. Let
 $K=K(G)$ be the kernel subgroup.
 By Lemma~\ref{sisis} we have  $L_{2\times 2'}(G/K)< L_{2\times 2'}(G)$. By the induction hypothesis, $G/K$ has nonsoluble length at most $L_{2\times 2'}(G/K)$. Since $K$ has nonsoluble length 1 by Lemma~\ref{l-kern}, the result follows.
\end{proof}
The proof of Theorem~\ref{t-2l} is now complete.
 \end{proof}

\begin{proof}[Proof of Corollary~\ref{c-fh}]  Let $h$ be the maximum Fitting height of soluble subgroups of a finite group $G$. We need to show that the non-soluble length $\lambda (G)$  is at most $h$. The result follows immediately from Theorem~\ref{t-2ll}, since obviously $L_{2\times 2'}(G)\leq h$.
\end{proof}

\begin{proof}[Proof of Theorem~\ref{t-2l}(b)]
Bounding the non-$p$-soluble length of a finite group in terms of the maximum $p$-length of its $p$-soluble subgroups can be achieved by similar arguments as for $2$-length. We give a  different proof here in order to achieve a better bound.

Recall that $p$ is an odd prime, and let $h$ be the maximum $p$-length of $p$-soluble subgroups of a finite group $G$. We need to prove that the non-$p$-soluble length $\lambda_p (G)$  is at most $h$.

Let $K=K_p(G)$ be the $p$-kernel subgroup of $G$.
It is sufficient to prove that the maximum $p$-length of $p$-soluble subgroups for $G/K$ is strictly smaller than for $G$. This will enable easy induction, since $\lambda _p(K)\leq 1$.
 Thus, the result will follow from the following proposition.

\begin{proposition}
Let $p$ be an odd prime.
 Let
$K=K_p(G)$ be the $p$-kernel subgroup of a finite group $G$.
Then the maximum $p$-length of $p$-soluble subgroups for $G/K$ is strictly smaller than for $G$.
\end{proposition}

\begin{proof}
We use induction on $|G|$.  The argument proving the basis of this induction is incorporated in the step of induction.
We can of course assume that $R_p(G)=1$.

Let $V =S_1\times\dots\times S_m$ be the socle of $G$ equal to a direct product of nonabelian simple groups $S_i$ of order divisible by $p$.
Let $M$ be a subgroup of $G$ containing $K$ such that $M/K$ is a $p$-soluble subgroup of $G/K$ with maximum possible $p$-length $k$. Our task is to find a $p$-soluble subgroup of $G$ with $p$-length higher than $k$.
 By induction we may assume that $G=M$, so that $G/K$ is $p$-soluble of $p$-length $k$. Let $T$ be a Sylow $2$-subgroup of $K$ and let $H=N_G(T)$. Then $H$ is obviously a $p$-soluble subgroup of $G$ such that $KH=G$ by the Frattini argument.

Let, without loss of generality, $\{S_1,\dots ,S_r\}$ be an orbit in the permutational action of $G$ on $\{S_1, \dots , S_m\}$ such that the image of $G$ in the action on this orbit also has $p$-length $k$; this image coincides with the image of $H$. In the group $G_1=(S_1\times\dots\times S_r)H$, the quotient
$G_1/K_p(G_1)$ is $p$-soluble of $p$-length $k$ by assumption. By induction we can assume that $G=G_1$, so that $V$ is a minimal normal subgroup in $G$, and acting by conjugation $G$ transitively permutes  $S_1,  S_2, \dots , S_m$.

Since $p\ne 2$ by hypothesis and $p$ divides $|S_1|$, we can use Thompson's theorem on normal $p$-complements \cite[Ch.~8]{go} to choose a characteristic subgroup $C_1$ of a Sylow $p$-subgroup $P_1$ of $S_1$ such that $N_{S_1}(C_1)$ does not have a normal  $p$-complement.
Choose elements $a_i\in G$ such that $S_i=S_1^{a_i}$ for $i=1,\dots , m$. Let $P_i=P_1^{a_i}$ and $C_i=C_1^{a_i}$.
Let $P=\prod P_i$
and $C=\prod C_i$. Then $P$ is a Sylow $p$-subgroup of $V$. Of course, every subgroup  $N_{S_i}(C_i)=N_{S_1}(C_1)^{a_i}$ also  does not have a normal  $p$-complement.

We claim that $N_G(C)\geq N_G(P)$. Indeed, if $x\in N_G(P)$ fixes $S_i$, then it fixes $P_i$, and therefore fixes $C_i$, which is characteristic in $P_i$. Now let $S_i^x=S_j$. Then we must have $P_i^{x}=P_j$, since $x\in N_G(P)$. Then $P_1^{a_ix}=P_1^{a_j}$, that is, $P_1^{a_ixa_j^{-1}}=P_1$, whence $C_1^{a_ixa_j^{-1}}=C_1$, since this is a characteristic subgroup of $P_1$. As a result, $C_i^x=C_1^{a_ix}=C_1^{a_j}=C_j$, which completes the proof of the claim.

By the Frattini argument we now have $VN_G(C)\geq V N_G(P) =G$, so that $VN_G(C)=G$.

Suppose that $N=N_G(C)$ is not $p$-soluble; then we can use the induction hypothesis as follows. Any non-$p$-soluble sections of $G$ are inside $V$,
and $N_V(C)=\prod N_{S_i}(C_i)$, so   the $N_{S_i}(C_i)$ are not $p$-soluble.  In the quotient $\bar N=N/R_p(N)$ by the $p$-soluble radical, the simple factors of the socle  $Soc(\bar N)$ are the images of some subgroups of the $N_{S_i}(C_i)$ (possibly, several subgroups in the same $N_{S_i}(C_i)$). The subgroups $N_{S_i}(C_i)$ are permuted by $N$, since $N$ permutes the $S_i$ and therefore must permute the $C_i$ being $N=N_G(C_1\times\dots\times C_m)$. Moreover, these permutational actions of $N$ are transitive. Hence the kernel $K_p(N)$ of the permutational action of $N$ on the set of simple factors of $Soc(\bar N)$ is contained in  $K\cap N$. Since $G/K=NK/K\cong N/(K\cap N)$, it follows that the maximum $p$-length of $p$-soluble subgroups for $ N/K_p(N)$ is at least the same as for $G/K$. Obviously, $N$ is a proper subgroup of $G$, so we can use induction for $N$, which also finishes the proof for~$G$.

Thus, we can assume that $N=N_G(C)$ is $p$-soluble. (This case also accounts for the basis of induction on $|G|$.) We claim that $O_{p',p}(N)\leq K$. Indeed, if $x\in O_{p',p}(N)$ but $x\not\in K$, then there is $i\in \{1,\dots, m\}$ such that $S_i^x\ne S_i$. Then the $S_i$-projection of $[x, N_{S_i}(C_i)]$ contains $N_{S_i}(C_i)$, which does not have a normal $p$-complement. But $[x,N_{S_i}(C_i)]\leq O_{p',p}(N)$, since $x\in O_{p',p}(N)$ and $N_{S_i}(C_i)\leq N$, a  contradiction.

Since $N/(N\cap K)\cong KN/K= VN/K=G/K$ and $O_{p',p}(N)\leq K\cap N$ as shown above, $N$ is a $p$-soluble subgroup of $p$-length greater than $k$.
\end{proof}
\renewcommand{\qedsymbol}{}
\end{proof}

\section{Groups with Sylow $p$-subgroups\\ in  products of soluble varieties and varieties of finite exponent}

We shall use the following fact several times.

\begin{lemma}\label{l-dlinn}
Let $G$ be a finite group and $P$ a Sylow $p$-subgroup of $G$.
Let $S_1\times\dots\times S_r$ be a normal subgroup of $G$ equal to a direct product of nonabelian simple groups $S_i$ the order of each of which is divisible by $p$. Let $N$ be  a normal subgroup of $P$ that has exponent $p^e$ and suppose that an element $g\in N$ has an orbit of length $p^e$ on the set $\{S_1,S_2,\ldots,S_r\}$ in the permutational action of $G$ induced by conjugation. Then this orbit is invariant under $N$.
\end{lemma}

\begin{proof}
Note that $P_i=P\cap S_i$ is a Sylow $p$-subgroup of $S_i$ for every $i$, and $P_i\ne 1$ by hypothesis.
 We assume without loss of generality that an orbit of $g\in N $ of length $p^e$  is $\{S_1,S_2,\ldots,S_{p^e}\}$. Let $a$ be any element of $N$. We claim that then  the set  $\{S_1,S_2,\ldots,S_{p^e}\}$ is invariant under $a$.  Suppose the opposite,  without loss of generality, $S_1^a\not\in \{S_1,S_2,\ldots,S_{p^e}\}$. Choose a nontrivial element $x\in P\cap S_1$ and consider the element
 $g[a,x]=g(x^{-1})^ax$, which also belongs to $N$ and therefore $(g(x^{-1})^ax)^{p^e}=1$.  Since  $(x^{-1})^a$ belongs to one of the $S_j$ outside the $g$-orbit  $\{S_1,S_2,\ldots,S_{p^e}\}$, its images under powers of $g$ stay outside this orbit and commute with the images of $x$ under powers of $g$. Therefore in the expansion
\begin{align*}
1&=(g(x^{-1})^ax)^{p^e}\\
  &=g^{p^e}((x^{-1})^ax)^{g^{p^e-1}} ((x^{-1})^ax)^{g^{p^e-2}} \cdots (x^{-1})^ax\\
  &=\big[ ((x^{-1})^a)^{g^{p^e-1}} ((x^{-1})^a)^{g^{p^e-2}} \cdots (x^{-1})^a\big]
\big[ x^{g^{p^e-1}} x^{g^{p^e-2}} \cdots x\big]
\end{align*}
both square brackets on the right must be trivial. This is a contradiction, since in the second bracket all elements $x^{g^i}$ are nontrivial in different subgroups $S_1,\dots ,S_{p_e}$.
Thus, the set  $\{S_1,S_2,\ldots,S_{p^e}\}$ is invariant under~$N$.
\end{proof}

In the proof of Theorem~\ref{t-exp-sol}  the following lemma will help to reduce exponents of normal subgroups in a Sylow $p$-subgroup.

\begin{lemma}\label{l-exp}
Let $P$ be a Sylow $p$-subgroup of a finite group $G$, and $N$ a normal subgroup of $P$ of exponent $p^e\neq 1$.
Then the $p^{e-1}$-th power $(N^{(e)})^{p^{e-1}}$ of the  $e$-th derived subgroup $N^{(e)}$ of $N$ is contained in the $p$-kernel subgroup $K_p(G)$. In particular, if $e=1$, then $N^{(1)}\leq K_p(G)$.
\end{lemma}

\begin{proof}
We can assume that $G$ has trivial $p$-soluble radical $R_p(G)=1$.
Let $V=
S_1\times S_2\times\cdots\times S_r $ be the socle of $G$, where the $ S_i$
are simple groups of order divisible by $p$.
            Acting on $V$ by conjugation the group $G$ acts by permutations on the set
 $\{S_1,S_2,\ldots,S_r\}$.
By Lemma~\ref{l-dlinn} the subgroup $N$ leaves invariant any orbit of length $p^e$ of some element of $N$.
 Since the Sylow $p$-subgroup of the symmetric group on $p^e$ symbols has derived length $e$, we obtain that the $e$-th derived subgroup $N^{(e)}$ normalizes those  factors of $V= S_1\times S_2\times\cdots\times S_r $ that are in an orbit of length $p^e$ of some element of $N$.  All other factors are in orbits of lengths dividing $p^{e-1}$ for all elements of $ N$, and therefore $N^{p^{e-1}}$ normalizes all such factors. Combining these facts, we obtain that $ (N^{(e)})^{p^{e-1}}\leq K_p(G)$
\end{proof}

The next lemma will help to tackle the derived length of normal subgroups  in a Sylow $p$-subgroup.

\begin{lemma}\label{l-abel} Let $G$ be a finite group and $P$ a Sylow $p$-subgroup of $G$, and let $A$ be an abelian normal subgroup of $P$.
Then $A\leq K_p(G)$ if $p\ne 2$, and $A^2\leq K(G)$ if $p=2$.
\end{lemma}

\begin{proof} We can assume that $G$ has trivial $p$-soluble radical $R_p(G)=1$. Let
$V=
S_1\times S_2\times\cdots\times S_r $ be the socle of $G$ equal to a direct product of
 simple groups $S_i$ of order divisible by $p$. Then $P_i=P\cap S_i\ne 1 $ is a Sylow $p$-subgroup of $S_i$ for every  $i$.
  We claim that $A^2\leq K_p(G)$. Suppose this is false. Then there exist an integer $m\geq 3$ and an element $a\in A$ which has an orbit of length $m$ in the set $\{S_1,S_2,\ldots,S_r\}$. Without loss of generality we assume that this orbit is $\{S_1, \ldots,S_m\}$. Choose $1\ne x\in P_1$ and look at the element $[x,a]=x^{-1}x^a\in S_1\times S_2$ with non-trivial projections both on $ S_1$ and $S_2$. Since both $a$ and $[x,a]$ belong to the abelian subgroup $A$, we conclude that $a$ and $[x,a]$ commute. On the other hand, $[x,a]^a$ has a non-trivial projection onto $S_3$ and therefore cannot be equal to $[x,a]$, a contradiction.
\end{proof}

\begin{proof}[Proof of  Theorem~\ref{t-exp-sol}] Recall that  $G$ is a finite group,  $P$ a Sylow $p$-subgroup of $G$, and  $P$ belongs to the variety ${\frak B}_{p^{a_1}}{\frak A}^{d_1}\cdots {\frak B}_{p^{a_n}}{\frak A}^{d_n}$; we need to bound
the non-$p$-soluble length $\lambda _p(G)$ in terms of $\sum a_i+\sum d_i$. We  use double induction: first by the number of factors in this product of varieties, and for a given number of factors secondary induction on the parameter  of the first nontrivial factor. The basis of the primary induction is a single variety factor when $P$ either  is soluble of derived length $d_1$, or has exponent $p^{a_1}$. In the first case the $p$-length of $p$-soluble subgroups of $G$ is at most $d_1$ by the theorems of Hall and Higman \cite{ha-hi} for $p\ne 2$, and Bryukhanova \cite{bry81} for $p=2$ (earlier a weaker bound for $p=2$ was obtained by Berger and Gross \cite{ber-gro}). In the second case  the $p$-length of $p$-soluble subgroups is at most $2a_1+1$  for $p\ne 2$ by the theorem of Hall and Higman \cite{ha-hi}, and at most $a_1$ for $p=2$  by Bryukhanova's theorem \cite{bry79} (earlier weaker bounds for $p=2$ were obtained by Hoare \cite{hoa} and  Gross \cite{gro}). In either case, then the result follows by Theorem~\ref{t-2l}.

When we perform the secondary induction on $d_1$ or $a_1$, we aim at killing off the first variety factor thus reducing the number of (nontrivial) variety factors,  possibly at the expense of increasing some of the remaining parameters $a_i, d_j$. These increases will happen in a controlled manner, so that the resulting sum of the $a_i, d_j$ will still be bounded in terms of the original sum, and the result will follow.

We can of course assume that all factors of our  product variety ${\frak B}_{p^{a_1}}{\frak A}^{d_1}\cdots {\frak B}_{p^{a_n}}{\frak A}^{d_n}$ are nontrivial, apart from possibly the first and the last ones, since ${\frak A}^{k}{\frak A}^{l}={\frak A}^{k+l}$ and
${\frak B}_{p^{k}}{\frak B}_{p^{l}}\subseteq {\frak B}_{p^{k+l}}$.

First assume that the first nontrivial factor of the product variety is indeed ${\frak B}_{p^{a_1}}$ for $a_1\geq 1$. Then
a Sylow $p$-subgroup $P$ has a normal subgroup $N$ of exponent  $p^e$ for $e\leq {a_1}$ such that $P/N$ belongs to the shorter product variety ${\frak A}^{d_1}\cdots {\frak B}_{p^{a_n}}{\frak A}^{d_n}$. By Lemma~\ref{l-exp} we have $(N^{(e)})^{p^{e-1}}\leq K_p(G)$. This means that the image of $P$ in $G/K_p(G)$ belongs to the product variety ${\frak B}_{p^{e-1}}{\frak A}^{d_1+e}\cdots {\frak B}_{p^{a_n}}{\frak A}^{d_n}$ with at most the same number of factors  as before and with a smaller value of the first parameter. The result now follows by induction, since $K_p(G)$ has non-$p$-soluble length at most 1, and the increase in the value of $d_1$ is bounded.
(Actually, it is easy to see that $G/K_{p,a_1}(G)$ has a Sylow $p$-subgroup in the variety ${\frak A}^{d_1+a_1(a_1+1)/2}\cdots {\frak B}_{p^{a_n}}{\frak A}^{d_n}$.)

We now consider the case where the first nontrivial factor of the product variety is actually ${\frak A}^{d_1}$ for $d_1\geq 1$ (that is, when $a_1=0$). Then a Sylow $p$-subgroup $P$ has a normal subgroup $N$ of derived length $d$ for $d\leq {d_1}$ such that $P/N$ belongs to the shorter product variety ${\frak B}_{p^{a_2}}{\frak A}^{d_2}\cdots {\frak B}_{p^{a_n}}{\frak A}^{d_n}$. Consider the case $p\ne 2$, where the argument is much simpler.
By Lemma~\ref{l-abel} we have  $N^{(d-1)}\leq K_p(G)$.
This means that the image of $P$ in $G/K_p(G)$ belongs to the product variety
with at most the same number of factors  as before and with a smaller value of the first parameter. The result now follows by induction, since $K_p(G)$ has non-$p$-soluble length at most 1.
(Actually, it is easy to see that $G/K_{p,d_1}(G)$ has a Sylow $p$-subgroup in the variety ${\frak B}_{p^{a_2}}{\frak A}^{d_2}\cdots {\frak B}_{p^{a_n}}{\frak A}^{d_n}$.)

The case $p=2$ requires more complicated arguments furnished by
the following lemma.

\begin{lemma}\label{l-cl2}
Let $H$ be a finite group and $P$ a Sylow $2$-subgroup of $H$.
Suppose that $Q$ is a normal subgroup of $P$, and let $d$ be the derived length of $Q$.

 {\rm (a)} Then $(Q^{(d-1)})^2\leq K(H)$.

 {\rm (b)} If $d\geq 2$, $(Q^{(d-1)})^2=1$, and $Q^{(d-1)}\not\leq K(H)$, then the image of $Q^{(d-2)}$ in $H/K(H)$ is nilpotent of class 2 and the image of $(Q^{(d-2)})^2$ is abelian.

   {\rm (c)} If $d\geq 2$, $(Q^{(d-1)})^2=1$,  $Q^{(d-2)}$ is nilpotent of class 2, and $Q^{(d-1)}\not\leq K(H)$, then $d=2$.
\end{lemma}

\begin{proof}
Part (a) follows from Lemma~\ref{l-abel}.

We firstly state some properties common to both parts (b) and (c). We can assume that $H$ has trivial soluble radical $R(H)=1$.
Let $V=S_1\times S_2\times\cdots\times S_r $ be the socle of $H$, where the $ S_i$
are
nonabelian simple groups. Then $P_i=P\cap S_i\ne 1$ is a Sylow $2$-subgroup of $S_i$ for every $i$. Acting on $V$ by conjugation the group $H$ acts by permutations on the set
 $\{S_1,S_2,\ldots,S_r\}$.
Since  $Q^{(d-1)}$ is of exponent $2$, every element  of $Q^{(d-1)}$ has  orbits of length $1$  or $2$ on the set $\{S_1, \dots , S_r\}$, and there is at least one orbit of length $2$ by the hypothesis $Q^{(d-1)}\not\leq K(H)$. Let $\{S_1, S_2\}$ be such an orbit. By Lemma~\ref{l-dlinn}, the set $\{S_1, S_2\}$ is invariant under $Q^{(d-1)}$.  Thus the set $\{S_1, \dots , S_r\}$ is a disjoint union of the set of fixed points of $Q^{(d-1)}$ and  two-element subsets that exhaust all two-element orbits of elements of $Q^{(d-1)}$. Since $Q^{(d-1)}$ is a normal subgroup, the group $Q$ permutes the two-element orbits of elements of $Q^{(d-1)}$.

We now prove (b).
Let $a\in Q^{(d-2)}$; we claim that $a^2$ leaves the orbit $\{S_1, S_2\}$ invariant. Assume the opposite; then   $\{S_1, S_2\}$, $\{S_1^a, S_2^a\}$, and $\{S_1^{a^2}, S_2^{a^2}\}$ are disjoint subsets. Choose a nontrivial element $x\in S_1\cap P$. Then $[x,a]\in Q^{(d-2)}$ and therefore, $[[x,a],a]\in Q^{(d-1)}$. We have $[[x,a],a]=(x^{-1})^ax(x^{-1})^ax^{a^2}$ with commuting factors $x\in S_1$, $(x^{-2})^a\in S_1^a$, and $x^{a^2}\in S_1^{a^2}$ in different two-element orbits of elements of $Q^{(d-1)}$. Then obviously $Q^{(d-1)}$ cannot centralize this product, a contradiction.

Furthermore, assuming that  $\{S_1, S_2\}$ is not invariant under an element $a\in Q^{(d-2)}$, we now claim that then the  $4$-point set $\{S_1, S_2, S_1^a, S_2^a\}$
is invariant under $Q^{(d-2)}$. Indeed, let $b\in Q^{(d-2)}$. By what was proved above, if this $4$-point set is not invariant under $b$, then we have disjoint subsets $\{S_1, S_2\}$, $\{S_1^a, S_2^a\}$, and $\{S_1^{b}, S_2^{b}\}$. Then again $[x,a]\in Q^{(d-2)}$ and therefore, $[[x,a],b]\in Q^{(d-1)}$. We have $[[x,a],b]=(x^{-1})^ax(x^{-1})^bx^{ab}$ with nontrivial factors $x\in S_1$, $(x^{-1})^a\in S_1^a$, and $(x^{-1})^b\in S_1^{b}$ in three different two-element orbits of elements of $Q^{(d-1)}$. Whichever orbit $x^{ab}$ belongs to, it cannot make the product to be  centralized  by $Q^{(d-1)}$, a contradiction.

As a result, the set $\{S_1, \dots , S_r\}$ is a disjoint union of  $Q^{(d-2)}$-invariant subsets of cardinalities $4$ or $2$ and the set of fixed points for $Q^{(d-1)}$. Then  the image of $Q^{(d-2)}$ in $H/K$ embeds in a direct product of Sylow $2$-subgroups of symmetric groups on $2$ or  $4$ symbols and the image of the abelian group $Q^{(d-2)}/Q^{(d-1)}$. Hence the   image of $Q^{(d-2)}$ in $H/K$ is nilpotent of class 2 and the image of $(Q^{(d-2)})^2$ in $H/K$ is abelian.

We now prove (c).
Suppose that $d\geq 3$. Let $a\in Q^{(d-3)}$, and let $\{S_1, S_2\}$ be one of two-element orbits for some elements of $Q^{(d-1)}$. We claim that $a^2$ leaves the orbit $\{S_1, S_2\}$ invariant. Assume the opposite; then   $\{S_1, S_2\}$, $\{S_1^a, S_2^a\}$, and $\{S_1^{a^2}, S_2^{a^2}\}$ are disjoint subsets. Choose a nontrivial element $x\in S_1\cap P$. Then $[x,a]\in Q^{(d-3)}$ and therefore, $[[x,a],a]\in Q^{(d-2)}$. We have $[[x,a],a]=(x^{-1})^ax(x^{-1})^ax^{a^2}$ with commuting factors $x\in S_1$, $(x^{-2})^a\in S_1^a$, and $x^{a^2}\in S_1^{a^2}$ in different two-element orbits of elements of $Q^{(d-1)}$. Then obviously $Q^{(d-1)}$ cannot centralize this product, which contradicts the hypothesis that  $Q^{(d-1)}=[Q^{(d-2)},Q^{(d-2)}]$ is in the centre of $Q^{(d-2)}$.

Furthermore, assuming that  $\{S_1, S_2\}$ is not invariant under an element $a\in Q^{(d-3)}$, we now claim that the $4$-point set $\{S_1, S_2, S_1^a, S_2^a\}$
is invariant under $Q^{(d-3)}$. Indeed, let $b\in Q^{(d-3)}$. By what was proved above, if this $4$-point set is not invariant under $b$, then we have disjoint subsets $\{S_1, S_2\}$, $\{S_1^a, S_2^a\}$, and $\{S_1^{b}, S_2^{b}\}$. Then again $[x,a]\in Q^{(d-3)}$ and therefore, $[[x,a],b]\in Q^{(d-2)}$. We have $[[x,a],b]=(x^{-1})^ax(x^{-1})^bx^{ab}$ with nontrivial factors $x\in S_1$, $(x^{-1})^a\in S_1^a$, and $(x^{-1})^b\in S_1^{b}$ in three different two-element orbits of elements of $Q^{(d-1)}$. Whichever orbit $x^{ab}$ belongs to, it cannot make the product to be  centralized  by $Q^{(d-1)}$, a contradiction with the hypothesis that  $Q^{(d-1)}=[Q^{(d-2)},Q^{(d-2)}]$ is in the centre of $Q^{(d-2)}$.

As a result, the set $\{S_1, \dots , S_r\}$ is a disjoint union of  $Q^{(d-3)}$-invariant subsets of cardinalities $4$ or $2$ and the set of fixed points for $Q^{(d-1)}$. Therefore the image of $Q^{(d-3)}$ in $H/K$ embeds in a direct product of Sylow $2$-subgroups of symmetric groups on $2$ or  $4$ symbols and the image of the metabelian  group $Q^{(d-3)}/Q^{(d-1)}$. Hence the   image of $Q^{(d-3)}$ in $H/K$ is metabelian, a contradiction with the hypothesis that $Q^{(d-1)}\not \leq K$. Thus, $d=2$.
\end{proof}

We now finish the proof of Theorem~\ref{t-exp-sol} in the remaining case where $p=2$ and the first variety factor is
${\frak A}^{d_1}$. Recall that then a Sylow $2$-subgroup $P$ has a normal subgroup $N$ of derived length $d$ for $d\leq {d_1}$ such that $P/N$ belongs to the shorter product variety ${\frak B}_{2^{a_2}}{\frak A}^{d_2}\cdots {\frak B}_{2^{a_n}}{\frak A}^{d_n}$.
By Lemma~\ref{l-cl2}(a) we have  $(N^{(d-1)})^2\leq K(G)$. If we also have  $N^{(d-1)}\leq K(G)$, then
the image of $P$ in $G/K(G)$ belongs to the product variety
with at most the same number of factors  as before and with a smaller value of the first parameter, and the result follows by induction, since $K(G)$ has nonsoluble length at most 1. If $d=1$, then in the quotient  $G/K(G)$  a Sylow $2$-subgroup  belongs to the shorter product variety ${\frak B}_{2^{a_2+1}}{\frak A}^{d_2}\cdots {\frak B}_{2^{a_n}}{\frak A}^{d_n}$, and again the result follows by induction.

Thus we can assume that $d\geq 2$ and $N^{(d-1)}\not\leq K(G)$.
If $N^{(d-1)}\leq K_2(G)$, then the result follows by induction, since the nonsoluble length of $K_2(G)$ is at most~2. Otherwise we can apply  Lemma~\ref{l-cl2}(b) to  $H=G/K(G)$, by which   the image of $N^{(d-2)}$ in $H/K(H)=G/K_2(G)$ is nilpotent of class 2 and the image of $(N^{(d-2)})^2$ is abelian. If $N^{(d-1)}\leq K_3(G)$, then the result follows by induction, since the nonsoluble length of $K_3(G)$ is at most 3.

Otherwise we can apply  Lemma~\ref{l-cl2}(c) to  $H=G/K_2(G)$, by which $d=2$. We now apply Lemma~\ref{l-abel} to $G/K_2(G)$ and the abelian subgroup $(N^{(d-2)})^2K_2(G)/K_2(G)$, by which $((N^{(d-2)})^2)^2\leq K_3(G)$; and since $N^{(d-2)}K_2(G)=NK_2(G)$, we obtain $N^4\leq (N^2)^2\leq K_3(G)$. But then a Sylow $2$-subgroup of $G/K_3(G)$ belongs
to the shorter product variety ${\frak B}_{2^{a_2+2}}{\frak A}^{d_2}\cdots {\frak B}_{2^{a_n}}{\frak A}^{d_n}$, and the result follows by induction, since the nonsoluble length of $K_3(G)$ is at most 3 and the increase in the value of $a_2$ is bounded.

(Actually, it is easy to see that $G/K_{3d_1}(G)$ has a Sylow $2$-subgroup in the variety ${\frak B}_{2^{a_2+2}}{\frak A}^{d_2}\cdots {\frak B}_{2^{a_n}}{\frak A}^{d_n}$.)
\end{proof}

\section{Correcting an error in an earlier paper }

If $w$ is a word in variables $x_1,\dots,x_t$ we think of
it primarily as a function of $t$ variables defined on any
given group $G$. The corresponding verbal subgroup $w(G)$
is the subgroup of $G$ generated by the values of $w$. A word $w$ is a multilinear commutator
if it can be written as a multilinear Lie monomial. Particular examples
of multilinear commutators are the derived words, defined by the equations:
$$\delta_0(x)=x,$$
$$\delta_k(x_1,\dots,x_{2^k})=
[\delta_{k-1}(x_1,\dots,x_{2^{k-1}}),
\delta_{k-1}(x_{2^{k-1}+1}\dots,x_{2^k})],$$
and the lower central words:
$$\gamma_1(x)=x,$$
$$\gamma_{k+1}(x_1,\dots,x_{k+1})=
[\gamma_{k}(x_1,\dots,x_{k}),x_{k+1}].$$
The following result was obtained in \cite{896}.

{\it Let $w$ be a multilinear commutator and $n$ a positive integer. Suppose that $G$ is a residually finite group in which every product of at most 896 $w$-values has order dividing $n$. Then the verbal subgroup $w(G)$ is locally finite.}

One key step in the proof is the following proposition (\cite[Proposition 3.3]{896}).

\begin{proposition}\label{ppp} Let $k$, $m$, $n$ be positive integers
and $G$ a finite group in which every product of 896
${\delta_k}$-commutators is of order dividing $n$. Assume that $G$ can be
generated by $m$ elements $g_1,g_2,\ldots,g_m$ such that each $g_i$ and all
commutators of the forms $[g,x]$ and $[g,x,y]$, where
$g\in\{g_1,g_2,\ldots,g_m\}$, $x,y\in G$, have orders dividing $n$.
Then the order of $G$ is bounded by a function depending only on $k,m,n$.
\end{proposition}

However the proof of the proposition given in \cite{896} contains an error. More specifically, the error is contained in the proof of Lemma 2.5 of \cite{896} so it is not clear at all whether the lemma is correct. Therefore amendments to the proof of the proposition must be made.  Corollary \ref{c-fh} enables one to make the necessary amendments in a relatively easy way. We will require the following fact which is straightforward from Corollary \ref{c-fh}.

\begin{lemma}\label{stan}  Let $h$ be the maximum Fitting height of soluble subgroups of a finite group~$G$. Then $G$ has a normal series of length at most $h^2+2h$ each of whose factors is either nilpotent or isomorphic to a direct product of nonabelian simple groups.
\end{lemma}

Proposition \ref{ppp} can now be proved as follows. By \cite[Lemma 2.4]{896}, there exists a number $h$ depending only on $n$ such that the Fitting height of any soluble subgroup of $G$ is at most $h$. Thus, by Lemma \ref{stan}, $G$ possesses a normal series
$$
G=G_1\geq G_2\geq\cdots\geq G_s\geq G_{s+1}=1
$$
such that each quotient $G_{i}/G_{i+1}$ is either nilpotent or is a Cartesian product of nonabelian simple groups and $s$ is bounded in terms of $n$ only.

By induction on $s$ we can assume that the order of $|G/G_s|$ is bounded by a function depending only on $k,m,n$. If $G_s$ is nilpotent, the result is immediate from \cite[Lemma 3.2]{896}. Otherwise $G_s$ is isomorphic to a Cartesian product of nonabelian simple groups. Since the order of $G/G_s$ is bounded by a function depending only on $k,m,n$, we deduce that $G_s$ can be generated by a bounded number, say $r$, of elements.  A result of Jones \cite{jones} says that any infinite family of finite simple groups generates the variety of all groups. It follows that up to isomorphism there exist only finitely many finite simple groups in which every ${\delta_k}$-commutator is of order dividing $n$. Let $N=N(k,n)$ be the maximum of the orders of these groups. Then $G_s$ is residually of order at most $N$. Since $G_s$ is $r$-generated, the number of distinct normal subgroups of index at most $N$ in $G_s$ is
$\{r,N\}$-bounded \cite[Theorem 7.2.9]{mhall}. Therefore $G_s$ has
$\{k,m,n\}$-bounded order. We conclude that $|G|$ is $\{k,m,n\}$-bounded, as required.

The problematic Lemma 2.5 of \cite{896} was later used in \cite{austral} and \cite{dms1}. In \cite{austral} it was used in the proof of the following result (\cite[Proposition 3.4]{austral}).
\medskip

{\it Let $e$ and $k$ be positive integers. Assume that $G$ is a finite group such that $x^e=1$ whenever $x$ is a $\delta_k$-commutator in $G$. Assume further that $P^{(k)}$ has exponent dividing $e$ for every $p\in\pi(G)$ and every Sylow $p$-subgroup $P$ of $G$. Then the exponent of $G^{(k)}$ is $(e,k)$-bounded.}
\medskip

In \cite{dms1} the lemma was used to prove that the following holds (\cite[Theorem A]{dms1}).\medskip

{\it Let $w$ be either the $n$-th Engel word or the word $[x^n,y_1,y_2,\dots,y_k]$. Assume that $G$ is a finite group in which any nilpotent subgroup generated by $w$-values has exponent dividing $e$. Then the exponent of the verbal subgroup $w(G)$ is bounded in terms of $e$ and $w$ only.}
\medskip

In both of the above results the Fitting heights of soluble subgroups in $G$ are bounded in terms of the respective parameters (cf. \cite[Lemma 2.7]{austral} and \cite[Lemma 2.9]{dms1}). Thus, we can employ our Corollary \ref{c-fh} and produce a proof of each of the above results that does not use the problematic lemma.

\end{document}